\documentclass{amsart}
\usepackage{amssymb}
\usepackage{amsfonts}

\newtheorem{theorem}{Theorem}
\theoremstyle{plain}

\newtheorem{definition}{Definition}

\newtheorem{lemma}{Lemma}

\newtheorem{remark}{Remark}

\numberwithin{equation}{section}
\input{tcilatex}

\begin{document}
\title[Hermite-Hadamard type inequalities via fractional integrals]{%
Hermite-Hadamard type inequalities for $s-$convex and $s-$concave functions
via fractional integrals}
\author{M.Emin \"{O}zdemir$^{\blacklozenge }$}
\address{$^{\blacklozenge }$ATAT\"{U}RK UNIVERSITY, K.K. EDUCATION FACULTY,
DEPARTMENT OF MATHEMATICS, 25240, CAMPUS, ERZURUM, TURKEY}
\email{emos@atauni.edu.tr}
\author{Merve Avc\i $^{\clubsuit \blacktriangledown }$}
\address{$^{\clubsuit }$ADIYAMAN UNIVERSITY, FACULTY OF SCIENCE AND ART,
DEPARTMENT OF MATHEMATICS, 02040, ADIYAMAN, TURKEY}
\email{mavci@posta.adiyaman.edu.tr}
\thanks{$^{\blacktriangledown }$Corresponding Author}
\author{Havva Kavurmaci$^{^{\blacklozenge }}$}

\begin{abstract}
New identity for fractional integrals have been defined. By using of this
identity, some new Hermite-Hadamard type inequalities for Riemann-Liouville
fractional integral have been developed. Our results have some relationships
with the result of Avci et al., proved in \cite[published in Appl. Math.
Comput., 217(2011) 5171-5176]{AKO}.
\end{abstract}

\maketitle

\section{introduction}

The following inequality, named Hermite-Hadamard inequality, is one of the
best known results in the literature.

\begin{theorem}
\label{teo 1.1} Let $f:I\subseteq 
%TCIMACRO{\U{211d} }%
%BeginExpansion
\mathbb{R}
%EndExpansion
\rightarrow 
%TCIMACRO{\U{211d} }%
%BeginExpansion
\mathbb{R}
%EndExpansion
$ be a convex function on an interval $I$ of real numbers and $a,b\in I$
with $a<b.$ Then the following inequality holds:%
\begin{equation*}
f\left( \frac{a+b}{2}\right) \leq \frac{1}{b-a}\int_{a}^{b}f(x)dx\leq \frac{%
f(a)+f(b)}{2}.
\end{equation*}
\end{theorem}

In \cite{HM}, Hudzik and Maligranda considered among others the class of
functions which are s-convex in the second sense.

\begin{definition}
\label{def 1.1} A function $f:%
%TCIMACRO{\U{211d} }%
%BeginExpansion
\mathbb{R}
%EndExpansion
^{+}\rightarrow 
%TCIMACRO{\U{211d} }%
%BeginExpansion
\mathbb{R}
%EndExpansion
,$ where $%
%TCIMACRO{\U{211d} }%
%BeginExpansion
\mathbb{R}
%EndExpansion
^{+}=[0,\infty ),$ is said to be $s-$convex in the second sense if%
\begin{equation*}
f\left( \alpha x+\beta y\right) \leq \alpha ^{s}f(x)+\beta ^{s}f(y)
\end{equation*}%
for all $x,y\in \lbrack 0,\infty ),$ $\alpha ,\beta \geq 0$ with $\alpha
+\beta =1$ and for some fixed $s\in (0,1].$ This class of s-convex functions
in the second sense is usually denoted by $K_{s}^{2}.$
\end{definition}

It can be easily seen that for $s=1$, $s-$convexity reduces to ordinary
convexity of functions defined on $[0,\infty )$.

In \cite{DF}, Dragomir and Fitzpatrick proved a variant of Hadamard's
inequality which holds for $s-$convex functions in the second sense.

\begin{theorem}
\label{teo 1.2} Suppose that $f:[0,\infty )\rightarrow \lbrack 0,\infty )$
is an $s-$convex functions in the second sense, where $s\in (0,1),$ and let $%
a,b\in \lbrack 0,\infty ),$ $a<b.$ If $f\in L^{1}[a,b],$ then the following
inequalities hold:%
\begin{equation}
2^{s-1}f\left( \frac{a+b}{2}\right) \leq \frac{1}{b-a}\int_{a}^{b}f(x)dx\leq 
\frac{f(a)+f(b)}{s+1}.  \label{1.1}
\end{equation}
\end{theorem}

The constant $k=\frac{1}{s+1}$ is the best possible in the second inequality
in (\ref{1.1}).

In \cite{KAO}, Kavurmaci et al. proved the following identity.

\begin{lemma}
\label{lem 1.1} Let $f:I\subset 
%TCIMACRO{\U{211d} }%
%BeginExpansion
\mathbb{R}
%EndExpansion
\rightarrow 
%TCIMACRO{\U{211d} }%
%BeginExpansion
\mathbb{R}
%EndExpansion
$ be a differentiable function on $I^{\circ },$where $a,b\in I$ with $a<b.$
If $f^{\prime }\in L[a,b],$ then the following equality holds:%
\begin{eqnarray*}
&&\frac{\left( b-x\right) f(b)+\left( x-a\right) f(a)}{b-a}-\frac{1}{b-a}%
\int_{a}^{b}f(u)du \\
&=&\frac{\left( x-a\right) ^{2}}{b-a}\int_{0}^{1}\left( t-1\right) f^{\prime
}\left( tx+\left( 1-t\right) a\right) dt+\frac{\left( b-x\right) ^{2}}{b-a}%
\int_{0}^{1}\left( 1-t\right) f^{\prime }\left( tx+\left( 1-t\right)
b\right) dt.
\end{eqnarray*}
\end{lemma}

In \cite{AKO}, Avci et al. obtained the following results by using the above
Lemma.

\begin{theorem}
\label{teo 1.3} Let $f:I\subset \lbrack 0,\infty )\rightarrow 
%TCIMACRO{\U{211d} }%
%BeginExpansion
\mathbb{R}
%EndExpansion
$ be a differentiable function on $I^{\circ }$ such that $f^{\prime }\in
L[a,b],$where $a,b\in I$ with $a<b.$ If $\left\vert f^{\prime }\right\vert $
is $s-$convex on $[a,b]$ for some fixed $s\in (0,1]$, then%
\begin{eqnarray}
&&\left\vert \frac{\left( b-x\right) f(b)+\left( x-a\right) f(a)}{b-a}-\frac{%
1}{b-a}\int_{a}^{b}f(u)du\right\vert  \label{1.2} \\
&\leq &\frac{1}{\left( s+1\right) \left( s+2\right) }\left\vert f^{\prime
}(x)\right\vert \left[ \frac{\left( x-a\right) ^{2}+\left( b-x\right) ^{2}}{%
b-a}\right]  \notag \\
&&+\frac{1}{\left( s+2\right) }\left[ \frac{\left( x-a\right) ^{2}}{b-a}%
\left\vert f^{\prime }(a)\right\vert +\frac{\left( b-x\right) ^{2}}{b-a}%
\left\vert f^{\prime }(b)\right\vert \right] .  \notag
\end{eqnarray}
\end{theorem}

\begin{theorem}
\label{teo 1.4} Let $f:I\subset \lbrack 0,\infty )\rightarrow 
%TCIMACRO{\U{211d} }%
%BeginExpansion
\mathbb{R}
%EndExpansion
$ be a differentiable function on $I^{\circ }$ such that $f^{\prime }\in
L[a,b],$where $a,b\in I$ with $a<b.$ If $\left\vert f^{\prime }\right\vert
^{q}$ is $s-$convex on $[a,b]$ for some fixed $s\in (0,1]$, $q>1$ with $%
\frac{1}{p}+\frac{1}{q}=1,$ then the following inequality holds:%
\begin{eqnarray}
&&\left\vert \frac{\left( b-x\right) f(b)+\left( x-a\right) f(a)}{b-a}-\frac{%
1}{b-a}\int_{a}^{b}f(u)du\right\vert  \label{1.3} \\
&\leq &\frac{\left( x-a\right) ^{2}}{b-a}\left( \frac{1}{p+1}\right) ^{\frac{%
1}{p}}\left[ \frac{\left\vert f^{\prime }(x)\right\vert ^{q}+\left\vert
f^{\prime }(a)\right\vert ^{q}}{s+1}\right] ^{\frac{1}{q}}  \notag \\
&&+\frac{\left( b-x\right) ^{2}}{b-a}\left( \frac{1}{p+1}\right) ^{\frac{1}{p%
}}\left[ \frac{\left\vert f^{\prime }(x)\right\vert ^{q}+\left\vert
f^{\prime }(b)\right\vert ^{q}}{s+1}\right] ^{\frac{1}{q}}.  \notag
\end{eqnarray}
\end{theorem}

\begin{theorem}
\label{teo 1.5} Suppose that all the assumptions of Theorem \ref{teo 1.4}
are satisfied. Then%
\begin{eqnarray}
&&\left\vert \frac{\left( b-x\right) f(b)+\left( x-a\right) f(a)}{b-a}-\frac{%
1}{b-a}\int_{a}^{b}f(u)du\right\vert  \label{1.4} \\
&\leq &\frac{\left( x-a\right) ^{2}}{b-a}\left( \frac{1}{2}\right) ^{1-\frac{%
1}{q}}\left( \left\vert f^{\prime }(x)\right\vert ^{q}\frac{1}{\left(
s+1\right) \left( s+2\right) }+\left\vert f^{\prime }(a)\right\vert ^{q}%
\frac{1}{s+2}\right) ^{\frac{1}{q}}  \notag \\
&&+\frac{\left( b-x\right) ^{2}}{b-a}\left( \frac{1}{2}\right) ^{1-\frac{1}{q%
}}\left( \left\vert f^{\prime }(x)\right\vert ^{q}\frac{1}{\left( s+1\right)
\left( s+2\right) }+\left\vert f^{\prime }(b)\right\vert ^{q}\frac{1}{s+2}%
\right) ^{\frac{1}{q}}.  \notag
\end{eqnarray}
\end{theorem}

\begin{theorem}
\label{teo 1.6} Let $f:I\subset \lbrack 0,\infty )\rightarrow 
%TCIMACRO{\U{211d} }%
%BeginExpansion
\mathbb{R}
%EndExpansion
$ be a differentiable function on $I^{\circ }$ such that $f^{\prime }\in
L[a,b],$where $a,b\in I$ with $a<b.$ If $\left\vert f^{\prime }\right\vert
^{q}$ is $s-$concave on $[a,b]$ for $q>1$ with $\frac{1}{p}+\frac{1}{q}=1,$
then the following inequality holds:%
\begin{eqnarray}
&&\left\vert \frac{\left( b-x\right) f(b)+\left( x-a\right) f(a)}{b-a}-\frac{%
1}{b-a}\int_{a}^{b}f(u)du\right\vert  \label{1.5} \\
&\leq &\frac{2^{\frac{s-1}{q}}}{\left( 1+p\right) ^{\frac{1}{p}}\left(
b-a\right) }\left\{ \left( x-a\right) ^{\alpha +1}\left\vert f^{\prime
}\left( \frac{x+a}{2}\right) \right\vert +\left( b-x\right) ^{\alpha
+1}\left\vert f^{\prime }\left( \frac{x+b}{2}\right) \right\vert \right\} . 
\notag
\end{eqnarray}
\end{theorem}

We give some necessary definitions and mathematical preliminaries of
fractional calculus theory which are used throughout this paper.

\begin{definition}
\label{def 1.2} Let $f\in L_{1}[a,b].$ The Riemann-Liouville integrals $%
J_{a^{+}}^{\alpha }(f)$ and $J_{b^{-}}^{\alpha }(f)$ of order $\alpha >0$
with $a\geq 0$ are defined by%
\begin{equation*}
J_{a^{+}}^{\alpha }f(x)=\frac{1}{\Gamma (\alpha )}\int_{a}^{x}\left(
x-t\right) ^{\alpha -1}f(t)dt,\text{ \ \ }x>a
\end{equation*}%
and%
\begin{equation*}
J_{b^{-}}^{\alpha }f(x)=\frac{1}{\Gamma (\alpha )}\int_{x}^{b}\left(
t-x\right) ^{\alpha -1}f(t)dt,\text{ \ \ }b>x
\end{equation*}%
where $\Gamma (\alpha )=\int_{0}^{\infty }e^{-t}u^{\alpha -1}du.$ Here $%
J_{a^{+}}^{0}f(x)=J_{b^{-}}^{0}f(x)=f(x).$
\end{definition}

In the case of $\alpha =1$, the fractional integral reduces to the classical
integral. Some recent results and properties concerning this operator can be
found in \cite{GM}-\cite{2}.

The main aim of this paper is to establish Hermite-Hadamard type
inequalities for $s-$convex and $s-$concave functions via Riemann-Liouville
fractional integral.

\section{Hermite-Hadamard type inequalities for fractional integrals}

In order to prove our main results we need the following Lemma.

\begin{lemma}
\label{lem 2.1} Let $f:I\subseteq 
%TCIMACRO{\U{211d} }%
%BeginExpansion
\mathbb{R}
%EndExpansion
\rightarrow 
%TCIMACRO{\U{211d} }%
%BeginExpansion
\mathbb{R}
%EndExpansion
$ be a differentiable function on $I^{\circ },$ the interior of $I,$ where $%
a,b\in I$ with $a<b.$ If $f^{\prime }\in L[a,b],$ then for all $x\in \lbrack
a,b]$ and $\alpha >0$ we have:%
\begin{eqnarray*}
&&\frac{\left( x-a\right) ^{\alpha }f(a)+\left( b-x\right) ^{\alpha }f(b)}{%
b-a}-\frac{\Gamma \left( \alpha +1\right) }{b-a}\left[ J_{x^{-}}^{\alpha
}f(a)+J_{x^{+}}^{\alpha }f(b)\right] \\
&=&\frac{\left( x-a\right) ^{\alpha +1}}{b-a}\int_{0}^{1}\left( t^{\alpha
}-1\right) f^{\prime }\left( tx+\left( 1-t\right) a\right) dt \\
&&+\frac{\left( b-x\right) ^{\alpha +1}}{b-a}\int_{0}^{1}\left( 1-t^{\alpha
}\right) f^{\prime }\left( tx+\left( 1-t\right) b\right) dt
\end{eqnarray*}%
where $\Gamma (\alpha )=\int_{0}^{\infty }e^{-t}u^{\alpha -1}du.$
\end{lemma}

\begin{proof}
By integration by parts, we can state%
\begin{eqnarray}
&&\int_{0}^{1}\left( t^{\alpha }-1\right) f^{\prime }\left( tx+\left(
1-t\right) a\right) dt  \label{2.1} \\
&=&\left. \left( t^{\alpha }-1\right) \frac{f\left( tx+\left( 1-t\right)
a\right) }{x-a}\right\vert _{0}^{1}-\int_{0}^{1}\alpha t^{\alpha -1}\frac{%
f\left( tx+\left( 1-t\right) a\right) }{x-a}dt  \notag \\
&=&\frac{f(a)}{x-a}-\frac{\alpha }{x-a}\int_{a}^{x}\left( \frac{u-a}{x-a}%
\right) ^{\alpha -1}\frac{f(u)}{x-a}du  \notag \\
&=&\frac{f(a)}{x-a}-\frac{\alpha \Gamma (\alpha )}{\left( x-a\right)
^{\alpha +1}}J_{x^{-}}^{\alpha }f(a)  \notag
\end{eqnarray}%
and%
\begin{eqnarray}
&&\int_{0}^{1}\left( 1-t^{\alpha }\right) f^{\prime }\left( tx+\left(
1-t\right) b\right) dt  \label{2.2} \\
&=&\left. \left( 1-t^{\alpha }\right) \frac{f\left( tx+\left( 1-t\right)
b\right) }{x-b}\right\vert _{0}^{1}-\int_{0}^{1}\alpha t^{\alpha -1}\frac{%
f\left( tx+\left( 1-t\right) b\right) }{x-b}dt  \notag \\
&=&\frac{f(b)}{b-x}-\frac{\alpha }{b-x}\int_{x}^{b}\left( \frac{u-b}{x-b}%
\right) ^{\alpha -1}\frac{f(u)}{x-b}du  \notag \\
&=&\frac{f(b)}{b-x}-\frac{\alpha \Gamma (\alpha )}{\left( b-x\right)
^{\alpha +1}}J_{x^{+}}^{\alpha }f(b).  \notag
\end{eqnarray}%
Multiplying the both sides of (\ref{2.1}) and (\ref{2.2}) by $\frac{\left(
x-a\right) ^{\alpha +1}}{b-a}$ and $\frac{\left( b-x\right) ^{\alpha +1}}{b-a%
},$ respectively, and adding the resulting identities we obtain the desired
result.
\end{proof}

\begin{theorem}
\label{teo 2.1} Let $f:I\subset \lbrack 0,\infty )\rightarrow 
%TCIMACRO{\U{211d} }%
%BeginExpansion
\mathbb{R}
%EndExpansion
$ be a differentiable function on $I^{\circ }$ such that $f^{\prime }\in
L[a,b],$where $a,b\in I$ with $a<b.$ If $\left\vert f^{\prime }\right\vert $
is $s-$convex on $[a,b]$ for some fixed $s\in (0,1]$ and $x\in \lbrack a,b],$
then the following inequality for fractional integrals with $\alpha >0$
holds:%
\begin{eqnarray*}
&&\left\vert \frac{\left( x-a\right) ^{\alpha }f(a)+\left( b-x\right)
^{\alpha }f(b)}{b-a}-\frac{\Gamma \left( \alpha +1\right) }{b-a}\left[
J_{x^{-}}^{\alpha }f(a)+J_{x^{+}}^{\alpha }f(b)\right] \right\vert \\
&\leq &\frac{\alpha }{\left( s+1\right) \left( \alpha +s+1\right) }\left[ 
\frac{\left( x-a\right) ^{\alpha +1}+\left( b-x\right) ^{\alpha +1}}{b-a}%
\right] \left\vert f^{\prime }(x)\right\vert \\
&&+\left[ \frac{1}{s+1}-\frac{\Gamma \left( \alpha +1\right) \Gamma \left(
s+1\right) }{\Gamma \left( \alpha +s+2\right) }\right] \left[ \frac{\left(
x-a\right) ^{\alpha +1}\left\vert f^{\prime }(a)\right\vert +\left(
b-x\right) ^{\alpha +1}\left\vert f^{\prime }(b)\right\vert }{b-a}\right]
\end{eqnarray*}%
where $\Gamma $ is Euler Gamma function.
\end{theorem}

\begin{proof}
From Lemma \ref{lem 2.1}, property of the modulus and using the $s-$%
convexity of $\left\vert f^{\prime }\right\vert $ we have%
\begin{eqnarray*}
&&\left\vert \frac{\left( x-a\right) ^{\alpha }f(a)+\left( b-x\right)
^{\alpha }f(b)}{b-a}-\frac{\Gamma \left( \alpha +1\right) }{b-a}\left[
J_{x^{-}}^{\alpha }f(a)+J_{x^{+}}^{\alpha }f(b)\right] \right\vert \\
&\leq &\frac{\left( x-a\right) ^{\alpha +1}}{b-a}\int_{0}^{1}\left\vert
t^{\alpha }-1\right\vert \left\vert f^{\prime }\left( tx+\left( 1-t\right)
a\right) \right\vert dt \\
&&+\frac{\left( b-x\right) ^{\alpha +1}}{b-a}\int_{0}^{1}\left\vert
1-t^{\alpha }\right\vert \left\vert f^{\prime }\left( tx+\left( 1-t\right)
b\right) \right\vert dt \\
&\leq &\frac{\left( x-a\right) ^{\alpha +1}}{b-a}\int_{0}^{1}\left(
1-t^{\alpha }\right) \left[ t^{s}\left\vert f^{\prime }(x)\right\vert
+\left( 1-t\right) ^{s}\left\vert f^{\prime }(a)\right\vert \right] dt \\
&&+\frac{\left( b-x\right) ^{\alpha +1}}{b-a}\int_{0}^{1}\left( 1-t^{\alpha
}\right) \left[ t^{s}\left\vert f^{\prime }(x)\right\vert +\left( 1-t\right)
^{s}\left\vert f^{\prime }(b)\right\vert \right] dt \\
&=&\frac{\left( x-a\right) ^{\alpha +1}}{b-a}\left\{ \int_{0}^{1}\left(
1-t^{\alpha }\right) t^{s}\left\vert f^{\prime }(x)\right\vert
dt+\int_{0}^{1}\left( 1-t^{\alpha }\right) \left( 1-t\right) ^{s}\left\vert
f^{\prime }(a)\right\vert dt\right\} \\
&&+\frac{\left( b-x\right) ^{\alpha +1}}{b-a}\left\{ \int_{0}^{1}\left(
1-t^{\alpha }\right) t^{s}\left\vert f^{\prime }(x)\right\vert
dt+\int_{0}^{1}\left( 1-t^{\alpha }\right) \left( 1-t\right) ^{s}\left\vert
f^{\prime }(b)\right\vert dt\right\} \\
&=&\frac{\alpha }{\left( s+1\right) \left( \alpha +s+1\right) }\left[ \frac{%
\left( x-a\right) ^{\alpha +1}+\left( b-x\right) ^{\alpha +1}}{b-a}\right]
\left\vert f^{\prime }(x)\right\vert \\
&&+\left[ \frac{1}{s+1}-\frac{\Gamma \left( \alpha +1\right) \Gamma \left(
s+1\right) }{\Gamma \left( \alpha +s+2\right) }\right] \left[ \frac{\left(
x-a\right) ^{\alpha +1}\left\vert f^{\prime }(a)\right\vert +\left(
b-x\right) ^{\alpha +1}\left\vert f^{\prime }(b)\right\vert }{b-a}\right] .
\end{eqnarray*}
We have used the fact that 
\begin{equation*}
\int_{0}^{1}\left( 1-t^{\alpha }\right) t^{s}dt=\frac{\alpha }{\left(
s+1\right) \left( \alpha +s+1\right) }
\end{equation*}%
and%
\begin{equation*}
\int_{0}^{1}\left( 1-t^{\alpha }\right) \left( 1-t\right) ^{s}dt=\left[ 
\frac{1}{s+1}-\frac{\Gamma \left( \alpha +1\right) \Gamma \left( s+1\right) 
}{\Gamma \left( \alpha +s+2\right) }\right]
\end{equation*}%
where $\beta $ is Euler Beta function defined by 
\begin{equation*}
\beta (x,y)=\int_{0}^{1}t^{x}\left( 1-t\right) ^{y}dt,\text{ \ \ \ }x,y>0
\end{equation*}%
and 
\begin{equation*}
\beta (x,y)=\frac{\Gamma (x)\Gamma (y)}{\Gamma (x+y)}.
\end{equation*}%
The proof is completed.
\end{proof}

\begin{remark}
\label{rem 2.1} In Theorem \ref{2.1}, if we choose $\alpha =1,$ we get the
inequality in (\ref{1.2}).
\end{remark}

\begin{theorem}
\label{teo 2.2} Let $f:I\subset \lbrack 0,\infty )\rightarrow 
%TCIMACRO{\U{211d} }%
%BeginExpansion
\mathbb{R}
%EndExpansion
$ be a differentiable function on $I^{\circ }$ such that $f^{\prime }\in
L[a,b],$where $a,b\in I$ with $a<b.$ If $\left\vert f^{\prime }\right\vert
^{q}$ is $s-$convex on $[a,b]$ for some fixed $s\in (0,1]$, $p,q>1$, $x\in
\lbrack a,b],$ then the following inequality for fractional integrals holds:%
\begin{eqnarray*}
&&\left\vert \frac{\left( x-a\right) ^{\alpha }f(a)+\left( b-x\right)
^{\alpha }f(b)}{b-a}-\frac{\Gamma \left( \alpha +1\right) }{b-a}\left[
J_{x^{-}}^{\alpha }f(a)+J_{x^{+}}^{\alpha }f(b)\right] \right\vert \\
&\leq &\left( \frac{\Gamma \left( 1+p\right) \Gamma \left( 1+\frac{1}{\alpha 
}\right) }{\Gamma \left( 1+p+\frac{1}{\alpha }\right) }\right) ^{\frac{1}{p}%
}\left\{ \frac{\left( x-a\right) ^{\alpha +1}}{b-a}\left( \frac{\left\vert
f^{\prime }(x)\right\vert ^{q}+\left\vert f^{\prime }(a)\right\vert ^{q}}{s+1%
}\right) ^{\frac{1}{q}}\right. \\
&&\left. +\frac{\left( b-x\right) ^{\alpha +1}}{b-a}\left( \frac{\left\vert
f^{\prime }(x)\right\vert ^{q}+\left\vert f^{\prime }(b)\right\vert ^{q}}{s+1%
}\right) ^{\frac{1}{q}}\right\}
\end{eqnarray*}%
where $\frac{1}{p}+\frac{1}{q}=1,$ $\alpha >0$ and $\Gamma $ is Euler Gamma
function.
\end{theorem}

\begin{proof}
From Lemma \ref{lem 2.1}, property of the modulus and using the H\"{o}lder
inequality we have%
\begin{eqnarray*}
&&\left\vert \frac{\left( x-a\right) ^{\alpha }f(a)+\left( b-x\right)
^{\alpha }f(b)}{b-a}-\frac{\Gamma \left( \alpha +1\right) }{b-a}\left[
J_{x^{-}}^{\alpha }f(a)+J_{x^{+}}^{\alpha }f(b)\right] \right\vert \\
&\leq &\frac{\left( x-a\right) ^{\alpha +1}}{b-a}\int_{0}^{1}\left\vert
t^{\alpha }-1\right\vert \left\vert f^{\prime }\left( tx+\left( 1-t\right)
a\right) \right\vert dt \\
&&+\frac{\left( b-x\right) ^{\alpha +1}}{b-a}\int_{0}^{1}\left\vert
1-t^{\alpha }\right\vert \left\vert f^{\prime }\left( tx+\left( 1-t\right)
b\right) \right\vert dt \\
&\leq &\frac{\left( x-a\right) ^{\alpha +1}}{b-a}\left\{ \left(
\int_{0}^{1}\left( 1-t^{\alpha }\right) ^{p}dt\right) ^{\frac{1}{p}}\left(
\int_{0}^{1}\left\vert f^{\prime }\left( tx+\left( 1-t\right) a\right)
\right\vert ^{q}dt\right) ^{\frac{1}{q}}\right\} \\
&&+\frac{\left( b-x\right) ^{\alpha +1}}{b-a}\left\{ \left(
\int_{0}^{1}\left( 1-t^{\alpha }\right) ^{p}dt\right) ^{\frac{1}{p}}\left(
\int_{0}^{1}\left\vert f^{\prime }\left( tx+\left( 1-t\right) b\right)
\right\vert ^{q}dt\right) ^{\frac{1}{q}}\right\} .
\end{eqnarray*}%
Since $\left\vert f^{\prime }\right\vert ^{q}$ is $s-$convex on $[a,b],$ we
get 
\begin{equation*}
\int_{0}^{1}\left\vert f^{\prime }\left( tx+\left( 1-t\right) a\right)
\right\vert ^{q}dt\leq \frac{\left\vert f^{\prime }(x)\right\vert
^{q}+\left\vert f^{\prime }(a)\right\vert ^{q}}{s+1},
\end{equation*}%
\begin{equation*}
\int_{0}^{1}\left\vert f^{\prime }\left( tx+\left( 1-t\right) b\right)
\right\vert ^{q}dt\leq \frac{\left\vert f^{\prime }(x)\right\vert
^{q}+\left\vert f^{\prime }(b)\right\vert ^{q}}{s+1}
\end{equation*}%
and by simple computation 
\begin{equation*}
\int_{0}^{1}\left( 1-t^{\alpha }\right) ^{p}dt=\frac{\Gamma \left(
1+p\right) \Gamma \left( 1+\frac{1}{\alpha }\right) }{\Gamma \left( 1+p+%
\frac{1}{\alpha }\right) }.
\end{equation*}%
Hence we have%
\begin{eqnarray*}
&&\left\vert \frac{\left( x-a\right) ^{\alpha }f(a)+\left( b-x\right)
^{\alpha }f(b)}{b-a}-\frac{\Gamma \left( \alpha +1\right) }{b-a}\left[
J_{x^{-}}^{\alpha }f(a)+J_{x^{+}}^{\alpha }f(b)\right] \right\vert \\
&\leq &\frac{\left( x-a\right) ^{\alpha +1}}{b-a}\left( \frac{\Gamma \left(
1+p\right) \Gamma \left( 1+\frac{1}{\alpha }\right) }{\Gamma \left( 1+p+%
\frac{1}{\alpha }\right) }\right) ^{\frac{1}{p}}\left( \frac{\left\vert
f^{\prime }(x)\right\vert ^{q}+\left\vert f^{\prime }(a)\right\vert ^{q}}{s+1%
}\right) ^{\frac{1}{q}} \\
&&+\frac{\left( b-x\right) ^{\alpha +1}}{b-a}\left( \frac{\Gamma \left(
1+p\right) \Gamma \left( 1+\frac{1}{\alpha }\right) }{\Gamma \left( 1+p+%
\frac{1}{\alpha }\right) }\right) ^{\frac{1}{p}}\left( \frac{\left\vert
f^{\prime }(x)\right\vert ^{q}+\left\vert f^{\prime }(b)\right\vert ^{q}}{s+1%
}\right) ^{\frac{1}{q}}
\end{eqnarray*}%
which completes the proof.
\end{proof}

\begin{remark}
\label{rem 2.2} In Theorem \ref{2.2}, if we choose $\alpha =1,$ we get the
inequality in (\ref{1.3}).
\end{remark}

\begin{theorem}
\label{teo 2.3} Let $f:I\subset \lbrack 0,\infty )\rightarrow 
%TCIMACRO{\U{211d} }%
%BeginExpansion
\mathbb{R}
%EndExpansion
$ be a differentiable function on $I^{\circ }$ such that $f^{\prime }\in
L[a,b],$where $a,b\in I$ with $a<b.$ If $\left\vert f^{\prime }\right\vert
^{q}$ is $s-$convex on $[a,b]$ for some fixed $s\in (0,1]$, $q\geq 1$, $x\in
\lbrack a,b],$ then the following inequality for fractional integrals holds:%
\begin{eqnarray*}
&&\left\vert \frac{\left( x-a\right) ^{\alpha }f(a)+\left( b-x\right)
^{\alpha }f(b)}{b-a}-\frac{\Gamma \left( \alpha +1\right) }{b-a}\left[
J_{x^{-}}^{\alpha }f(a)+J_{x^{+}}^{\alpha }f(b)\right] \right\vert \\
&\leq &\left( \frac{\alpha }{\alpha +1}\right) ^{1-\frac{1}{q}} \\
&&\times \left\{ \frac{\left( x-a\right) ^{\alpha +1}}{b-a}\left( \frac{%
\alpha }{\left( s+1\right) \left( \alpha +s+1\right) }\left\vert f^{\prime
}(x)\right\vert ^{q}+\left[ \frac{1}{s+1}-\frac{\Gamma \left( \alpha
+1\right) \Gamma \left( s+1\right) }{\Gamma \left( \alpha +s+2\right) }%
\right] \left\vert f^{\prime }(a)\right\vert ^{q}\right) ^{\frac{1}{q}%
}\right. \\
&&\left. +\frac{\left( b-x\right) ^{\alpha +1}}{b-a}\left( \frac{\alpha }{%
\left( s+1\right) \left( \alpha +s+1\right) }\left\vert f^{\prime
}(x)\right\vert ^{q}+\left[ \frac{1}{s+1}-\frac{\Gamma \left( \alpha
+1\right) \Gamma \left( s+1\right) }{\Gamma \left( \alpha +s+2\right) }%
\right] \left\vert f^{\prime }(b)\right\vert ^{q}\right) ^{\frac{1}{q}%
}\right\}
\end{eqnarray*}%
where $\alpha >0$ and $\Gamma $ is Euler Gamma function.
\end{theorem}

\begin{proof}
From Lemma \ref{lem 2.1}, property of the modulus and using the power-mean
inequality we have%
\begin{eqnarray}
&&  \label{2.3} \\
&&\left\vert \frac{\left( x-a\right) ^{\alpha }f(a)+\left( b-x\right)
^{\alpha }f(b)}{b-a}-\frac{\Gamma \left( \alpha +1\right) }{b-a}\left[
J_{x^{-}}^{\alpha }f(a)+J_{x^{+}}^{\alpha }f(b)\right] \right\vert  \notag \\
&\leq &\frac{\left( x-a\right) ^{\alpha +1}}{b-a}\int_{0}^{1}\left\vert
t^{\alpha }-1\right\vert \left\vert f^{\prime }\left( tx+\left( 1-t\right)
a\right) \right\vert dt+\frac{\left( b-x\right) ^{\alpha +1}}{b-a}%
\int_{0}^{1}\left\vert 1-t^{\alpha }\right\vert \left\vert f^{\prime }\left(
tx+\left( 1-t\right) b\right) \right\vert dt  \notag \\
&\leq &\frac{\left( x-a\right) ^{\alpha +1}}{b-a}\left\{ \left(
\int_{0}^{1}\left( 1-t^{\alpha }\right) dt\right) ^{1-\frac{1}{q}}\left(
\int_{0}^{1}\left( 1-t^{\alpha }\right) \left\vert f^{\prime }\left(
tx+\left( 1-t\right) a\right) \right\vert ^{q}dt\right) ^{\frac{1}{q}%
}\right\}  \notag \\
&&+\frac{\left( b-x\right) ^{\alpha +1}}{b-a}\left\{ \left(
\int_{0}^{1}\left( 1-t^{\alpha }\right) dt\right) ^{1-\frac{1}{q}}\left(
\int_{0}^{1}\left( 1-t^{\alpha }\right) \left\vert f^{\prime }\left(
tx+\left( 1-t\right) b\right) \right\vert ^{q}dt\right) ^{\frac{1}{q}%
}\right\} .  \notag
\end{eqnarray}%
Since $\left\vert f^{\prime }\right\vert ^{q}$ is $s-$convex on $[a,b],$ we
get%
\begin{eqnarray}
&&  \label{2.4} \\
\int_{0}^{1}\left( 1-t^{\alpha }\right) \left\vert f^{\prime }\left(
tx+\left( 1-t\right) a\right) \right\vert ^{q}dt &\leq &\int_{0}^{1}\left(
1-t^{\alpha }\right) \left[ t^{s}\left\vert f^{\prime }(x)\right\vert
^{q}+\left( 1-t\right) ^{s}\left\vert f^{\prime }(a)\right\vert ^{q}\right] 
\notag \\
&=&\frac{\alpha }{\left( s+1\right) \left( \alpha +s+1\right) }\left\vert
f^{\prime }(x)\right\vert ^{q}+\left[ \frac{1}{s+1}-\frac{\Gamma \left(
\alpha +1\right) \Gamma \left( s+1\right) }{\Gamma \left( \alpha +s+2\right) 
}\right] \left\vert f^{\prime }(a)\right\vert ^{q}  \notag
\end{eqnarray}%
and 
\begin{eqnarray}
&&  \label{2.5} \\
\int_{0}^{1}\left( 1-t^{\alpha }\right) \left\vert f^{\prime }\left(
tx+\left( 1-t\right) b\right) \right\vert ^{q}dt &\leq &\int_{0}^{1}\left(
1-t^{\alpha }\right) \left[ t^{s}\left\vert f^{\prime }(x)\right\vert
^{q}+\left( 1-t\right) ^{s}\left\vert f^{\prime }(b)\right\vert ^{q}\right] 
\notag \\
&=&\frac{\alpha }{\left( s+1\right) \left( \alpha +s+1\right) }\left\vert
f^{\prime }(x)\right\vert ^{q}+\left[ \frac{1}{s+1}-\frac{\Gamma \left(
\alpha +1\right) \Gamma \left( s+1\right) }{\Gamma \left( \alpha +s+2\right) 
}\right] \left\vert f^{\prime }(a)\right\vert ^{q}.  \notag
\end{eqnarray}%
If we use (\ref{2.4}) and (\ref{2.5}) in (\ref{2.3}), we obtain the desired
result.
\end{proof}

\begin{remark}
\label{rem 2.3} In Theorem \ref{2.3}, if we choose $\alpha =1,$ we get the
inequality in (\ref{1.4}).
\end{remark}

\begin{theorem}
\label{teo 2.4} Let $f:I\subset \lbrack 0,\infty )\rightarrow 
%TCIMACRO{\U{211d} }%
%BeginExpansion
\mathbb{R}
%EndExpansion
$ be a differentiable function on $I^{\circ }$ such that $f^{\prime }\in
L[a,b],$where $a,b\in I$ with $a<b.$ If $\left\vert f^{\prime }\right\vert
^{q}$ is $s-$concave on $[a,b]$ for some fixed $s\in (0,1]$, $q>1,$ $x\in
\lbrack a,b],$ then the following inequality for fractional integrals holds:%
\begin{eqnarray*}
&&\left\vert \frac{\left( x-a\right) ^{\alpha }f(a)+\left( b-x\right)
^{\alpha }f(b)}{b-a}-\frac{\Gamma \left( \alpha +1\right) }{b-a}\left[
J_{x^{-}}^{\alpha }f(a)+J_{x^{+}}^{\alpha }f(b)\right] \right\vert \\
&\leq &\left( \frac{\Gamma \left( 1+p\right) \Gamma \left( 1+\frac{1}{\alpha 
}\right) }{\Gamma \left( 1+p+\frac{1}{\alpha }\right) }\right) ^{\frac{1}{p}}%
\frac{2^{\frac{s-1}{q}}}{b-a}\left\{ \left( x-a\right) ^{\alpha
+1}\left\vert f^{\prime }\left( \frac{x+a}{2}\right) \right\vert +\left(
b-x\right) ^{\alpha +1}\left\vert f^{\prime }\left( \frac{x+b}{2}\right)
\right\vert \right\}
\end{eqnarray*}%
where $\frac{1}{p}+\frac{1}{q}=1,$ $\alpha >0$ and $\Gamma $ is Euler Gamma
function.
\end{theorem}

\begin{proof}
From Lemma \ref{lem 2.1}, property of the modulus and using the H\"{o}lder
inequality we have%
\begin{eqnarray}
&&  \label{2.6} \\
&&\left\vert \frac{\left( x-a\right) ^{\alpha }f(a)+\left( b-x\right)
^{\alpha }f(b)}{b-a}-\frac{\Gamma \left( \alpha +1\right) }{b-a}\left[
J_{x^{-}}^{\alpha }f(a)+J_{x^{+}}^{\alpha }f(b)\right] \right\vert  \notag \\
&\leq &\frac{\left( x-a\right) ^{\alpha +1}}{b-a}\int_{0}^{1}\left\vert
t^{\alpha }-1\right\vert \left\vert f^{\prime }\left( tx+\left( 1-t\right)
a\right) \right\vert dt  \notag \\
&&+\frac{\left( b-x\right) ^{\alpha +1}}{b-a}\int_{0}^{1}\left\vert
1-t^{\alpha }\right\vert \left\vert f^{\prime }\left( tx+\left( 1-t\right)
b\right) \right\vert dt  \notag \\
&\leq &\frac{\left( x-a\right) ^{\alpha +1}}{b-a}\left\{ \left(
\int_{0}^{1}\left( 1-t^{\alpha }\right) ^{p}dt\right) ^{\frac{1}{p}}\left(
\int_{0}^{1}\left\vert f^{\prime }\left( tx+\left( 1-t\right) a\right)
\right\vert ^{q}dt\right) ^{\frac{1}{q}}\right\}  \notag \\
&&+\frac{\left( b-x\right) ^{\alpha +1}}{b-a}\left\{ \left(
\int_{0}^{1}\left( 1-t^{\alpha }\right) ^{p}dt\right) ^{\frac{1}{p}}\left(
\int_{0}^{1}\left\vert f^{\prime }\left( tx+\left( 1-t\right) b\right)
\right\vert ^{q}dt\right) ^{\frac{1}{q}}\right\} .  \notag
\end{eqnarray}%
Since $\left\vert f^{\prime }\right\vert ^{q}$ is $s-$concave on $[a,b],$
using the inequality (\ref{1.1}), we have%
\begin{equation}
\int_{0}^{1}\left\vert f^{\prime }\left( tx+\left( 1-t\right) a\right)
\right\vert ^{q}dt\leq 2^{s-1}\left\vert f^{\prime }\left( \frac{x+a}{2}%
\right) \right\vert ^{q}  \label{2.7}
\end{equation}%
and%
\begin{equation}
\int_{0}^{1}\left\vert f^{\prime }\left( tx+\left( 1-t\right) b\right)
\right\vert ^{q}dt\leq 2^{s-1}\left\vert f^{\prime }\left( \frac{x+b}{2}%
\right) \right\vert ^{q}.  \label{2.9}
\end{equation}%
From (\ref{2.6})-(\ref{2.9}), we complete the proof.
\end{proof}

\begin{remark}
\label{rem 2.4} In Theorem \ref{2.4}, if we choose $\alpha =1,$ we get the
inequality in (\ref{1.5}).
\end{remark}

\end{document}